\documentclass[12pt]{amsart}

\usepackage{amssymb,stmaryrd}

\usepackage{mathtools}

\usepackage{enumerate}
\usepackage{color}

 \newcommand{\la}{\mathfrak{g}}

\newcommand{\e}{\mathfrak{e}(3)}

\newcommand{\h}{\mathfrak{h}}

  \theoremstyle{definition}
  \newtheorem{definition}{Definition}[section]
  \newtheorem{example}[definition]{Example}
  \newtheorem{remark}[definition]{Remark}

  \theoremstyle{plain}
  \newtheorem{lemma}[definition]{Lemma}
  \newtheorem{proposition}[definition]{Proposition}
  \newtheorem{theorem}[definition]{Theorem}

\title[Abelian extensions of $D_N$ in $E_{N+1}$ ]{Classification of embeddings of abelian extensions of $D_n$ into $E_{n+1}$}

\begin{document}

\author[]{Andrew Douglas}

\author[]{Delaram Kahrobaei}
\author[]{Joe Repka}

\address[]{CUNY Graduate Center and New York City College of Technology, City University of New York, USA}
\email{adouglas2@gc.cuny.edu}
\email{dkahrobaei@gc.cuny.edu}
\address[]{ Department of Mathematics, University of Toronto, Canada}
\email{repka@math.toronto.edu}


\begin{abstract}
An abelian extension of the special orthogonal Lie algebra $D_n$ is a nonsemisimple Lie algebra $D_n \inplus V$, where $V$ is a finite-dimensional representation of $D_n$, with the understanding that $[V,V]=0$.   
We determine all abelian extensions of $D_n$ that may be embedded into the exceptional Lie algebra $E_{n+1}$, $n=5, 6$, and $7$.  We then classify  these embeddings, up to inner automorphism.   As an application, we also consider the restrictions of irreducible representations of $E_{n+1}$ to $D_n \inplus V$, and discuss which of these restrictions are or are not indecomposable.
\end{abstract}

\maketitle

\section{Introduction}
 
Let $V$ be a finite-dimensional representation of the orthogonal Lie algebra $D_n$.  Then, we may define an abelian extension of $D_n$ by $D_n \inplus V$:  The representation action of $D_n$ on $V$  defines a multiplication between  elements of $D_n$ and $V$, and $V$ is regarded as an abelian algebra:  $[V,V]=0$.  
 
 In this article, we determine all abelian extensions of $D_n$ that may be embedded into the exceptional Lie algebra $E_{n+1}$, $n=5, 6$, and $7$ (Section \ref{sec6}).  We then classify these embeddings, up to inner automorphism (Section \ref{sec7}).   As an application, we also consider the restrictions to $D_n \inplus V$ of irreducible representations of $E_{n+1}$, for each abelian extension and each embedding (Section \ref{sec8}).  

We begin in Section \ref{sec2} with motivation.  Section \ref{sec3} records relevant information about the objects of study in this article:  $D_n$ and $E_{n+1}$, $n=5, 6$, $7$.  Section \ref{sec4} presents pertinent notation and terminology.  In Section \ref{sec5} we describe the classification of  embeddings of $D_n$ into $E_{n+1}$, up to inner automorphism, which will be employed in the following sections.

\section{Motivation}\label{sec2}

The examination of embeddings of semisimple Lie algebras into complex semisimple Lie algebras is an interesting, important, and widely studied problem.  The foundational work of Dynkin \cite{dynkin} is perhaps the most widely cited.  Dynkin's work \cite{dynkin}, together with a recent article by Minchenko \cite{min}, classifies the semisimple subalgebras 
of the exceptional Lie algebras, up to inner automorphism.  From a classification of subalgebras, one may obtain a classification of embeddings in a straightforward manner.  In particular, a subalgebra $\la'$ of $\la$ naturally corresponds to one embedding, and may (or may not) correspond to more inequivalent embeddings if $\la'$ has outer automorphisms.

The current paper extends previous research by the likes of Dynkin and Minchenko by considering embeddings of {\it nonsemisimple} Lie algebras into  exceptional Lie algebras.  It also extends and generalizes  previous work of Douglas and Repka, which examined specific important examples of embeddings of nonsemisimple Lie algebras into classical Lie algebras.

For instance, in \cite{dr3}, the authors classified the embeddings of
the Euclidean algebra $\mathfrak{e}(3)\cong \mathfrak{so}(3,\mathbb{C}) \inplus \mathbb{C}^3$ into $\mathfrak{so}(5,\mathbb{C})$, up to inner automorphism.  As another example, the authors  classified the embeddings of $\mathfrak{e}(3)$ into $\mathfrak{sl}(4,\mathbb{C})$, up to inner automorphism \cite{dr1}.

Little is known about the indecomposable representations of abelian extensions $D_n \inplus V$.  The current research promises to shed some light on the representation theory of these nonsemisimple Lie algebras.   In particular,   as an application of the above mentioned classification,  the authors will examine the irreducible representations of $E_{n+1}$ restricted to $D_n \inplus V$, for each embedding and each abelian extension.  

The authors have employed this technique in previous research.  In \cite{dr3}, for instance,  Douglas and Repka showed that  each irreducible representation of $B_2$ remains indecomposable upon restriction to $\mathfrak{e}(3)$ $\cong$ $\mathfrak{so}(3,\mathbb{C}) \inplus \mathbb{C}^3$ with respect to any embedding of $\mathfrak{e}(3)$ into $B_2$,  hence creating a large family of indecomposable $\e$-representations.   In \cite{dr1} and \cite{dr2}   Douglas and Repka examined the irreducible representations of $A_3$ restricted to $\mathfrak{e}(3)$ with respect to any embedding;  these $\mathfrak{e}(3)$-representations were generally not indecomposable.

\section{The Lie algebras $D_n$ and $E_{n+1}$}\label{sec3}

The special orthogonal algebra $D_n$ is the Lie algebra of complex $2n\times 2n$ matrices $N$ satisfying $N^{tr}=-N$.  The dimension of $D_n$ is $2n^2-n$ and its rank is $n$.    
The Lie group corresponding to  $D_n$ arises naturally as  the symmetry group of  a real projective space over $\mathbb{R}$.

Besides the classical Lie algebras, which include $D_n$, there are five exceptional Lie algebras, three of which are $E_6$, $E_7$, and $E_8$.  These exceptional Lie algebras  were first explicitly constructed by Cartan in his $1894$ thesis, although their existence was proposed several years earlier by Wilhelm Killing \cite{baez}.

Like $D_n$, the exceptional Lie algebras  $E_6$, $E_7$, and $E_8$  have a close connection to  the Riemannian geometry of projective spaces (for details, we refer the reader to \cite{baez}).  The algebras $E_6$, $E_7$, and $E_8$ have ranks $6$, $7$, and $8$, and dimensions $78$, $133$, and $248$, respectively.

Let $\la$ denote $D_n$, $E_6$, $E_7$, or $E_8$.  
Let $k=n, 6, 7$, or $8$ for $D_n$, $E_6$, $E_7$, or $E_8$, respectively.  We may define $\la$ by a 
set of generators $\{H_i, X_i, Y_i\}_{1\leq i \leq k}$ together with the Chevalley-Serre relations \cite{humphreys}:
\begin{equation}\label{cbd}
\begin{array}{llllllllllll}
\displaystyle [H_i,H_j]=0, &[H_i,X_j]= M^{\la}_{ji} X_j,  \\
\displaystyle [H_i,Y_j]= -M^{\la}_{ji} Y_j,& [X_i,Y_j]= \delta_{ij} H_i, \\
\displaystyle  (\text{ad} ~X_i)^{1-M^{\la}_{ji}}(X_{j}) =0,  & (\text{ad} ~Y_i)^{1-M^{\la}_{ji}}(Y_j)=0,& \text{when}~i\neq j,
\end{array}
\end{equation}
where $1 \leq i,j \leq k$, and  $M^{\la}$ is the Cartan matrix of $\la$  (see \cite{humphreys}).   The $X_i$, for $1\leq i \leq k$, correspond to the simple roots.

There are numerous calculations throughout this article involving products of elements of $E_{n+1}$, $n=5, 6, 7$.  These calculations may be carried out with the Chevalley-Serre relations \eqref{cbd}, or,  more efficiently, using the computer algebra system GAP \cite{gap}.

The Dynkin diagrams of $D_n$, $E_6$, $E_7$, and $E_8$, indicating the numbering of simple roots, are given in Figure \ref{ddd}.  For future reference, we also list in Table \ref{exb} the maximal dimension of an abelian 
subalgebra of $E_6$, $E_7$, and $E_8$ from \cite{mal}.

\begin{table} [!ht]\renewcommand{\arraystretch}{1.2} \caption{Maximal abelian subalgebras of $E_6$, $E_7$, and $E_8$. \cite{mal}} \label{exb}\begin{center}
\begin{tabular}{|c|c|c|clclcl} 
\hline
Exceptional Lie Algebra &  Maximal Dimension of Abelian Subalgebra \\
\hline \hline
$E_6$& $16$\\ 
\hline 
$E_7$  & $27$\\
\hline 
$E_8$ & $36$\\
\hline
\end{tabular}\end{center}
\end{table}

\begin{figure} [!ht]
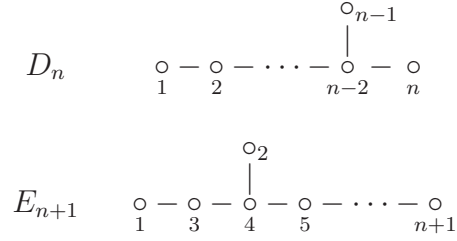
\caption{Dynkin diagrams of $D_n$, $E_{n+1}$.}\label{ddd}
 \begin{align*}\begin{array}{ccc}
 D_n  &&\underset{\mathclap{1}}{\circ} - \underset{\mathclap{2}}{\circ} - \dotsb - \underset{\mathclap{{n-2}}}{\overset{\overset{\textstyle\circ_{\mathrlap{n-1}}}{\textstyle\vert}}{\circ}} \,-\, \underset{\mathclap{{n}}}{\circ} \\\\
E_{n+1}  &&\underset{\mathclap{1}}{\circ} - \underset{\mathclap{3}}{\circ}                 - \underset{\mathclap{4}}{\overset{\overset{\textstyle\circ_{\mathrlap{2}}}{\textstyle\vert}}{\circ}}                - \underset{\mathclap{5}}{\circ} - \dotsb - \underset{\mathclap{n+1}}{\circ}  \\\\
%
\end{array}
\end{align*}
\end{figure}

We now consider the representations of $
\la = 
D_n$ 
or 
$E_{n+1}$, with   
$k$ as above.

For $i=1,...,k$, and $\h$ the Cartan subalgebra of $\la$ with basis $H_1,...,H_k$, define
$\alpha_i, ~ \lambda_i \in \h^*$ by $\alpha_i(H_j)= M^{\la}_{ji}$, and $\lambda_i(H_j)=\delta_{ij}$.  The $\lambda_i$ are the {\it fundamental weights}, and their indexing corresponds with that  of the  Dynkin diagram of type $D_n$, $E_6$, $E_7$, or $E_8$, respectively, in  Figure \ref{ddd}.

For each
$\lambda= m_1 \lambda_1+...+m_k \lambda_k  \in \h^*$, with nonnegative 
integers $m_1,...,m_k$, 
there exists a finite-dimensional,
irreducible $\la$-module 
with highest weight $\lambda$, 
denoted $V_{\la}(\lambda)$, and every finite-dimensional irreducible $\la$-module is of this form, for some $\lambda$.  The representations $V_{\la}(\lambda_i)$ for $1\leq i \leq k$ are the {\it fundamental 
representations}.  A useful formula for the dimensions of the fundamental representations of $D_n$ is as follows [\cite{hu}, Theorem $13.10$]:
\begin{equation}
\dim (V_{D_n}(\lambda_m) )=  \begin{cases} \binom{2n}{m}, & \text{if} ~~ 1 \leq m< n-1;\\
2^{n-1}, & \text{if} ~~ m= n-1,n.
\end{cases}
\end{equation}
The representation $V_{\la}(\lambda)$ is realized as the quotient of $\mathcal{U}(\la)$ by the left ideal, $J_{\lambda}$,
generated by $X_i, H_i- \lambda(H_i)$, $Y_{i}^{1+\lambda(H_i)}$, $1\leq i \leq k$ (here the action of $\mathcal{U}(\la)$ on itself and on 
$V_{\la}(\lambda)$ is given by left multiplication).  We will denote the element $1+J_\lambda$ of  the quotient $V_{\la}(\lambda)$ by 
$\widetilde{u}$.  Then one can show that $V_{\la}(\lambda)$ is generated by $\{ Y_{i_1} \cdots Y_{i_l} \widetilde{u}: l \in \mathbb{N}_0, i_1,...,
i_l \in \{1,...,k\} \}$.  The weight of   $Y_{i_1} \cdots Y_{i_l} \widetilde{u}$ is $\lambda- \Sigma_{j=1}^l \alpha_{i_j}$.

\section{Additional definitions and notation}\label{sec4}

The following definitions and notation will be used in this article.  In the definitions, let $\la$ and $\mathfrak{g}'$ be simple Lie algebras.

\begin{itemize}
\item Let $X_{a_i}$ correspond to a simple root of $E_{n+1}$, for $ 1\leq a_i \leq n+1$, where $n=5, 6$, or $7$.  We then define
\[
X_{a_1,a_2,a_3,...,a_k} \equiv [[...[[X_{a_{1}},X_{a_{2}}],X_{a_{3}}],...],X_{a_{k}}].
\]
$Y_{a_1,a_2,a_3,...,a_k}$ is defined analogously.

\item   Let $\varphi:  D_{n} \hookrightarrow E_{n+1}$ be an embedding, and let $W \in E_{n+1}$.     Then, we define $[W]_{\varphi(D_n)}$ to be the $D_n$-representation generated  by $W$ with respect to the adjoint action of $\varphi(D_n)$.  

\item Let  $W_1$,..., $W_m$ $\in E_{n+1}$, for $n=5, 6$, or $7$. Then, 
$\langle W_1,..., W_m \rangle_{E_{n+1}}$
is the subalgebra of $E_{n+1}$ generated by $W_1$,...,$W_m$.

\item Let $V$ be a finite-dimensional representation of $D_n$.
Then, a {\it lift} of the embedding $\varphi: D_n \hookrightarrow E_{n+1}$ to $D_n \inplus V$ is an embedding $\widetilde{\varphi}^V: D_n \inplus V\hookrightarrow E_{n+1}$  such that $\widetilde{\varphi}^V$ restricted to $D_n$ is equal to $\varphi$.  That is,   $\widetilde{\varphi}^V |_{D_n}= \varphi$.

\item Let $\varphi$ and $\varrho$ be Lie algebra embeddings of $\la'$ into $\mathfrak{g}$.  Then $\varphi$ and $\varrho$ are {\it equivalent} if there is an inner automorphism $\rho: \mathfrak{g} \rightarrow \mathfrak{g}$ such that $\varphi= \rho \circ 
\varrho$, and we write
\[
\varphi \sim \varrho.
\]
Hence, our classification in this article is up to equivalence.
\item Two embeddings $\varphi$ and $\varrho$  of $\la'$ into $\la_{}$ are \emph{linearly equivalent} if for any finite-dimensional  representation $V$ of $\la$,  $V|_{\varphi(\mathfrak{\la'})}$ and $V|_{\varrho(\mathfrak{\la'})}$ are isomorphic representations of $\la'$, and we write
\[
\varphi \sim_L \varrho.
\]
Clearly equivalence implies linear equivalence, but the converse is not in general true.   We define equivalence  of subalgebras much as we did for embeddings.
\item Two subalgebras $\la'$ and $\la''$ of $\la$ are equivalent if there is an inner automorphism $\rho$ or $\la$
such that $\rho(\la')=\la''$.

\item A subalgebra of $\la$ is {\it regular} if it is normalized by a Cartan subalgebra of $\la$.  
\end{itemize}

 \section{Embeddings of $D_{n}$ into $E_{n+1}$}\label{sec5}

In this section we describe the classification of  embeddings of $D_n$ into $E_{n+1}$, up to equivalence, which is given
in Theorem \ref{dymin}.  We begin by constructing two embeddings of $D_n$ into $E_{n+1}$.

With respect to the generators defined in Section \ref{sec3}, we may ``naturally" embed $D_{n}$ into $E_{n+1}$ for $n=5, 6,$ or $7$ as follows:
\begin{equation}
\begin{array}{llll}\label{emb}
\varphi_n: &D_n &\hookrightarrow& E_{n+1} \\
& H_{n+1-i} & \mapsto & H_{i+1}\\
& X_{n+1-i} & \mapsto & X_{i+1}\\
& Y_{n+1-i} & \mapsto & Y_{i+1},
\end{array}
\end{equation}
where $1\leq i\leq n$.

Each of these embeddings may be  visualized as arising from a ``natural" subgraph of the Dynkin diagram of $E_{n+1}$, for $n=5,6,$ or $7$, which is isomorphic to the Dynkin diagram of $D_n$  (see Figure \ref{ddd}).  

Consider the transposition $\sigma =(n-1,n)$  of the symmetric group $\mathcal{S}_n$, and define an outer automorphism of $D_n$ as follows:
\begin{equation}
\begin{array}{llll}\label{emb2}
\rho_n: &D_n &\rightarrow& D_{n} \\
& H_{i} & \mapsto & H_{\sigma(i)}\\
& X_{i} & \mapsto & X_{\sigma(i)}\\
& Y_{i} & \mapsto & Y_{\sigma(i)},\\
\end{array}
\end{equation}
where $1\leq i\leq n$.    The outer automorphisms of $D_{n}$ correspond to the automorphisms of its Dynkin diagram [\cite{fh}, Proposition $D. 40$], so that $\rho_n$ is the only outer automorphism of $D_{n}$, $n>4$,  up to inner automorphism.  

Using the above outer automorphism of $D_n$, define a second embedding of $D_n$ into $E_{n+1}$:
\begin{equation}
\varrho_n \equiv \varphi_n \circ \rho_n :  D_n \hookrightarrow E_{n+1}.
\end{equation}

The two embeddings $\varphi_n$ and $\varrho_n$ are the basis of the classification of embeddings of $D_n$ into $E_{n+1}$.  The classification is  a consequence of the following three theorems.

\begin{theorem}\cite{dynkin, min} \label{dymina}
There is a unique $D_n$ subalgebra inside $E_{n+1}$, up to equivalence, for $n=5,6$, and $7$.  Further, in each case, the subalgebra is regular.
\end{theorem}

\begin{theorem}[\cite{min}, Theorem $2$] \label{dymin2}
Let $\varphi$ and $\varrho$ be embeddings of $D_n$ into $E_{n+1}$, for $n=5,6$, and $7$, then
 \begin{equation}\label{vccd}
 \varphi \sim_L \varrho \Leftrightarrow V_{E_{n+1}}(\lambda_{n+1})|_{\varphi(D_n)}\cong V_{E_{n+1}}(\lambda_{n+1})|_{\varrho(D_n)}.
 \end{equation}
\end{theorem}

Interestingly, Theorem \ref{dymin2}, from \cite{min}, is stronger than a corresponding result found in Dynkin's classic paper [\cite{dynkin}, Theorem $1.3$].  In \cite{dynkin}, one needs to consider $V_{E_{n+1}}(\lambda_{n+1})$, and $V_{E_{n+1}}(\lambda_{1})$, for $n=6, 7$, to establish linear equivalence.  Note that Theorem \ref{dymin2}, although found in \cite{min}, is attributed to Losev \cite{losev}.
\begin{theorem}[\cite{min}, Theorem $3$] \label{dymin2b}
Let  $\varphi$ and $\varrho$ be embeddings of the simple Lie algebra $\la'$ into  
the 
simple Lie algebra $\mathfrak{g}$ such that $\varphi(\mathfrak{\la'})$ and $\varrho(\mathfrak{\la'})$ are regular subalgebras.  Then
 \begin{equation}\label{vcc}
 \varphi \sim_L \varrho \Leftrightarrow  \varphi \sim \varrho.
 \end{equation}
\end{theorem}

\begin{theorem}\label{dymin}
The  embedding $\varphi_7$ is the unique embedding of $D_7$ into $E_{8}$, up to equivalence.
 There are exactly two embeddings of  $D_5$ into $E_{6}$ and of $D_6$ into $E_{7}$, up to equivalence: $\varphi_n$ and $\varrho_n$, for $n=5$ or $6$, respectively.
\end{theorem}
\begin{proof}
We first show that $\varphi_n: D_n \hookrightarrow E_{n+1}$ and $\varrho_n: D_n \hookrightarrow E_{n+1}$ are equivalent if and only if $n=7$.  Consider the following decompositions.
   \begin{equation}\label{wq1}
\begin{array}{llllllll}
 V_{E_6}(\lambda_6)|_{\varphi(D_5)} 
&\cong&   V_{D_5}(0) \oplus  V_{D_5}(\lambda_1)\oplus V_{D_5}(\lambda_5),\\
V_{E_6}(\lambda_6)|_{\varrho (D_5)}
&\cong  & V_{D_5}(0) \oplus  V_{D_5}(\lambda_1)\oplus V_{D_5}(\lambda_4).
\end{array} 
 \end{equation}
$V_{E_6}(\lambda_6)|_{\varphi(D_5)} \ncong  V_{E_6}(\lambda_6)|_{\varrho (D_5)}$ implies $\varphi_5 \nsim_L \varrho_5$, and hence 
    \begin{equation}\label{w1}
 \varphi_5 \nsim \varrho_5.
 \end{equation}
 Also,
  \begin{equation}\label{wq2}
\begin{array}{llllllll}
 V_{E_7}(\lambda_7)|_{\varphi(D_6)} 
&\cong&   2V_{D_6}(\lambda_1)\oplus V_{D_6}(\lambda_6),\\
V_{E_7}(\lambda_7)|_{\varrho (D_6)}
&\cong  &  2V_{D_6}(\lambda_1)\oplus V_{D_6}(\lambda_5).
\end{array} 
 \end{equation}
$V_{E_7}(\lambda_7)|_{\varphi(D_6)} \ncong  V_{E_7}(\lambda_7)|_{\varrho (D_6)}$  implies $\varphi_6 \nsim_L \varrho_6$, and hence 
  \begin{equation}\label{w2}
 \varphi_6 \nsim \varrho_6.
 \end{equation}
 Finally,
 \begin{equation}\label{wq3}
\begin{array}{llllllllllllll}
 V_{E_8}(\lambda_8)|_{\varphi(D_7)} \cong  V_{E_8}(\lambda_8)|_{\varrho (D_7)}&\cong    V_{D_7}(\lambda_2) \oplus V_{D_7}(\lambda_1)  \oplus 
  V_{D_7}(\lambda_6)\\
 &\oplus V_{D_7}(\lambda_7) \oplus 
 V_{D_7}(\lambda_1)\oplus V_{D_7}(0).
\end{array} 
 \end{equation}
 $V_{E_8}(\lambda_8)|_{\varphi(D_7)} \cong  V_{E_8}(\lambda_8)|_{\varrho (D_7)}$  
 implies $\varphi_7 \sim_L \varrho_7$ (Theorem \ref{dymin2}), and hence, by Theorem  \ref{dymin2b},
 \begin{equation}\label{w3}
 \varphi_7 \sim \varrho_7.
 \end{equation}
 Since, by Theorem \ref{dymina}, there is only one $D_n$ subalgebra in $E_{n+1}$, up to equivalence, the only possible embedding inequivalent to $\varphi_n$ is formed from composing $\varphi_n$ with an outer automorphism of $D_n$, of which there is only one for $n>4$, namely $\rho_n$ \cite{fh}.  In other words, the only possible embedding inequivalent to $\varphi_n$ is $\varrho_n$.  The result then follows from Eqs. \eqref{w1}, \eqref{w2}, and \eqref{w3}.
 \end{proof}

\section{Abelian extensions of ${D}_n$ in $E_{n+1}$}\label{sec6}

In this section we determine the abelian extensions of $D_n$ that may be embedded into $E_{n+1}$, $n=5, 6$, $7$.   The results are summarized in Propositions \ref{mhh}, \ref{mhhh}, and \ref{mhhhh}.

\subsection{Abelian extensions of  ${D}_7$ in $E_{8}$}
We begin by calculating the decomposition of $E_8$ with respect to the adjoint action of $\varphi_7(D_7)$:
\begin{equation}\label{hhhhj}
\begin{array}{lllllllllllll}
E_8&\cong_{\varphi_7(D_7)}& [  X']_{\varphi_7(D_7)}&\oplus&  [ X''' ]_{\varphi_7(D_7)} &\oplus&  [ Y_{1} ]_{\varphi_7(D_7)} &\oplus\\
&&  [ X'']_{\varphi_7(D_7)} &\oplus&  [ Y' ]_{\varphi_7(D_7)}&\oplus& 
[ H ]_{\varphi_7(D_7)}\\
&  \cong_{\varphi_7(D_7)} & V_{D_7}(\lambda_2) &\oplus& V_{D_7}(\lambda_1) &\oplus& V_{D_7}(\lambda_6)&\oplus  \\
&& V_{D_7}(\lambda_7) &\oplus& V_{D_7}(\lambda_1)&\oplus&  V_{D_7}(0),
\end{array}
\end{equation}
where
\begin{equation}\label{ghyz}
\begin{array}{llll}
\displaystyle  X'&=&X_{4,5,6,7,8,2,3,4,5,6,7},\\ 
\displaystyle  X'' &=&-X_{3,4,2,1,5,4,3,6,5,4,7,2,6,5,8,7,6,4,5,3,4,2}, \\  
\displaystyle  X'''&=&  X_{8,7,6,5,4,3,2,1,4,5,6,7,3,4,5,6,2,4,5,3,4,2,1,3, 4,5,6,7,8},\\
\displaystyle Y'&=& -Y_{5,4,2,3,6,4,1,3,5,4,7,2,6,5,4,3,1},\\ 
\displaystyle H&=&  4H_1+5H_2+7H_3+10H_4+\\
\displaystyle && 8H_5+6H_6+4H_7+2H_8.
\end{array}
\end{equation}
Note that 
\begin{equation} \label{idd}
\varphi_7(D_7) = [X']_{\varphi_7(D_7)} \cong V_{D_7}(\lambda_2).
\end{equation}
Since $\varphi_7$ is the only embedding of $D_7$ into $E_8$, up to equivalence (Theorem \ref{dymin}), determining which abelian extensions of $D_7$ may be embedded into $E_8$
amounts
to determining which $E_8$-subspaces among $[ X''' ]_{\varphi_7(D_7)}$, $[ Y_{1} ]_{\varphi_7(D_7)}$, $[ X'']_{\varphi_7(D_7)}$, $[ Y' ]_{\varphi_7(D_7)}$, and  
$[ H ]_{\varphi_7(D_7)}$ in Eq. \eqref{hhhhj}, or direct sums of these subspaces, are abelian subalgebras of $E_8$.
\begin{lemma}\label{abal}
The $E_8$-subspaces  $[ X''']_{\varphi_7(D_7)}$,   $[ Y']_{\varphi_7(D_7)}$, and   $[ H ]_{\varphi_7(D_7)}$   are abelian subalgebras of $E_8$.  
\end{lemma}
\begin{proof}
Let $X'''_{a_{1},a_{2},...,a_{m}}=  [Y_{a_n},[Y_{a_{n-1}},...,[Y_{a_2},[Y_{a_1},X''']]...]]$;  then the following is a basis of  $V_{D_7}(\lambda_1)\cong [ X''']_{\varphi_7(D_7)}$:
 \begin{equation}\label{asgg}
\begin{array}{lllllll}
X''', &X'''_{8},& X'''_{8,7}, \\
X'''_{8,7,6}, & X'''_{8,7,6,5}, & X'''_{8,7,6,5,4},\\
 X'''_{8,7,6,5,4,2}, & X'''_{8,7,6,5,4,3}, & X'''_{8,7,6,5,4,2,3},\\
 X'''_{8,7,6,5,4,2,3,4}, & X'''_{8,7,6,5,4,2,3,4,5}, & X'''_{8,7,6,5,4,2,3,4,5,6},\\
  X'''_{8,7,6,5,4,2,3,4,5,6,7}, & X'''_{8,7,6,5,4,2,3,4,5,6,7,8}.
\end{array}
\end{equation}
Direct calculation shows
\begin{equation}\label{pjii}
[[ X''' ]_{\varphi_7(D_7)}, X''']=0.
\end{equation}
Eq. \eqref{pjii} together with the Jacobi identity imply
\begin{equation}
[[ X''' ]_{\varphi_7(D_7)}, [ X''']_{\varphi_7(D_7)}]=0.
\end{equation}
Hence, $[ X''' ]_{\varphi_7(D_7)}$ is an abelian subalgebra of $E_8$.   In a similar fashion, we show $[ Y']_{\varphi_7(D_7)}$ is an abelian subalgebra of $E_8$.   $[ H ]_{\varphi_7(D_7)}$  is $1$-dimensional, hence is  necessarily abelian. 
\end{proof}

\begin{lemma}\label{notab}
The subspaces $[\alpha X'''+ \beta Y']_{\varphi_7(D_7)}$, with $\alpha, \beta \in \mathbb{C}^*$,  $[X''']_{\varphi_7(D_7)}$ $\oplus$ $[Y']_{\varphi_7(D_7)}$, 
$[X''']_{\varphi_7(D_7)} \oplus [H]_{\varphi_7(D_7)}$, $[H]_{\varphi_7(D_7)} \oplus [Y']_{\varphi_7(D_7)}$,  $[X''']_{\varphi_7(D_7)} \oplus [Y']_{\varphi_7(D_7)} \oplus [H]_{\varphi_7(D_7)}$,  $[Y_1]_{\varphi_7(D_7)}$, and $[X'']_{\varphi_7(D_7)}$ are not  abelian subalgebras of $E_8$. 
\end{lemma}
\begin{proof}
Note that $Y_8 \in \varphi_7(D_7)$, so that $[Y_8,  \alpha X'''+ \beta Y'] \in [\alpha X'''+ \beta Y']_{\varphi_7(D_7)}$.  We have 
\begin{equation}
[\alpha X'''+ \beta Y', [Y_8,  \alpha X'''+ \beta Y'] ] = -2\alpha \beta X' \neq 0. 
\end{equation}
Thus, $[\alpha X'''+ \beta Y']_{\varphi_7(D_7)}$ is not abelian.  In a similar manner, we show that $[X''']_{\varphi_7(D_7)} \oplus [Y']_{\varphi_7(D_7)}$, 
$[X''']_{\varphi_7(D_7)} \oplus [H]_{\varphi_7(D_7)}$, $[H]_{\varphi_7(D_7)} \oplus [Y']_{\varphi_7(D_7)}$, and $[X''']_{\varphi_7(D_7)} \oplus [Y']_{\varphi_7(D_7)} \oplus [H]_{\varphi_7(D_7)}$ are not abelian.

Finally, 
$[Y_1]_{\varphi_7(D_7)}$ and $[X'']_{\varphi_7(D_7)}$ are not  abelian subalgebras of $E_8$ by Table \ref{exb} since 
$\dim(V_{\varphi(D_7)}(\lambda_6))$ $=$
$\dim(V_{\varphi(D_7)}(\lambda_7))=64$.
\end{proof}

Note the following isomorphisms of $\varphi(D_7)$-modules:
\begin{equation}\label{labb}
\begin{array}{llllllll}
\displaystyle [H]_{\varphi_7(D_7)} &\cong& V_{D_7}(0) &\cong&  \mathbb{C},\\
\displaystyle [X''']_{\varphi_7(D_7)} &\cong& V_{D_7}(\lambda_1)&\cong& \mathbb{C}^{14},\\
\displaystyle [Y']_{\varphi_7(D_7)} &\cong& V_{D_7}(\lambda_1)&\cong& \mathbb{C}^{14}.
\end{array}
\end{equation}
The irrep $V_{D_{7}}(\lambda_{1})$ is the ``standard'' representation of $D_{7} \cong \mathfrak{so}(14,\mathbb{C})$ on 
${\Bbb C}^{14}$ by matrix multiplication.  We shall write ${\Bbb C}^{14}$ to refer to this irrep, 
and ${\Bbb C}$ to refer to the trivial irrep $V_{D_{7}}(0)$.

Lemmas \ref{abal} and \ref{notab}  then give the following proposition.
\begin{proposition}\label{mhh}
The only abelian extensions of $D_7$ that may be embedded into $E_8$ are the following:
\begin{equation}
\begin{array}{ll}
D_7 \inplus \mathbb{C} ~ \text{and} ~ D_7 \inplus \mathbb{C}^{14}.
\end{array}
\end{equation} 
The abelian subspaces of $E_{8}$ which are isomorphic to ${\Bbb C}$ with respect to the adjoint action of 
$\varphi_7(D_7)$ have (highest weight) vector $\alpha  H$, for $\alpha \in \mathbb{C}^*$.  Those 
which are isomorphic to $\mathbb{C}^{14}$ with respect to the adjoint action of $\varphi_7(D_7)$   
have highest weight vector 
$\alpha  X'''$ or $\alpha  Y'$, for $\alpha \in \mathbb{C}^*$.  
\end{proposition}

\subsection{Abelian extensions of $D_6$ in $E_7$}

We begin by calculating the decomposition of $E_7$ with respect to the adjoint action of $\varphi_6(D_6)$ and $\varrho_6(D_6)$, respectively:
\begin{equation}\label{hhhh}
\begin{array}{lllllllll}
E_7 &
\cong_{\varphi_6(D_6)}&  [ X']_{\varphi_6(D_6)} &\oplus& [ X'' ]_{\varphi_6(D_6)} &\oplus& [ Y_1 ]_{\varphi_6(D_6)} &\oplus& \\
&& [ Y' ]_{\varphi_6(D_6)} &\oplus& [ X''' ]_{\varphi_6(D_6)}& \oplus &
 [ H]_{\varphi_6(D_6)} \\
&\cong_{\varphi_6(D_6)} &V_{D_6}(\lambda_2) &\oplus&     2V_{D_6}(\lambda_5) &\oplus& 3 V_{D_6}(0),
\end{array}
\end{equation}
\begin{equation}\label{hhhhp}
\begin{array}{lllllllllll}
E_7 &
\cong_{\varrho_6(D_6)}&  [ X']_{\varrho_6(D_6)} &\oplus& [ X'' ]_{\varrho_6(D_6)} &\oplus& [ Y_1 ]_{\varrho_6(D_6)} &\oplus& \\
&& [ Y' ]_{\varrho_6(D_6)} &\oplus& [ X''' ]_{\varrho_6(D_6)}& \oplus& 
 [ H]_{\varrho_6(D_6)} \\
&\cong_{\varrho_6(D_6)} &V_{D_6}(\lambda_2) &\oplus  &   2V_{D_6}(\lambda_6)& \oplus &3 V_{D_6}(0),
\end{array}
\end{equation}
where
\begin{equation}\label{ghyy}
\begin{array}{lllllll}
\displaystyle  X'&=& X_{6,7,5,4,3,2,4,5,6},\\ 
\displaystyle  X''&=& X_{7,6,5,4,3,2,4,5,6,1,3,4,5,2,4,3}, \\  
\displaystyle  X'''&=& -X_{7,6,5,4,3,2,4,5,6,1,3,4,5,2,4,3,1},\\ 
\displaystyle  Y'&= &-Y_{7,6,5,4,3,2,4,5,6,1,3,4,5,2,4,3,1}, \\ 
\displaystyle  H&= &2H_1+2H_2+3H_3+4H_4+3H_5+2H_6+H_7.
\end{array}
\end{equation}
Note that 
\begin{equation} \label{idddd}
\begin{array}{lllll}
\varphi_6(D_6) &=& [X']_{\varphi_6(D_6)} &\cong& V_{D_6}(\lambda_2),\\
\varrho_6(D_6) &=& [X']_{\varrho_6(D_6)} &\cong& V_{D_6}(\lambda_2).
\end{array}
\end{equation}
Since $\varphi_6$ and $\varrho_6$ are the only embeddings of $D_6$ into $E_7$, up to equivalence (Theorem \ref{dymin}), 
determining which abelian extensions of $D_6$ may be embedded into $E_7$ amounts to determining which $E_7$-subspaces in Eq. \eqref{hhhh} or \eqref{hhhhp} (excluding $[X']_{\varphi_6(D_6)}$ and $[X']_{\varrho_6(D_6)}$), or direct sums of these subspaces, are abelian subalgebras of $E_7$.

Eqs. \eqref{hhhh} and \eqref{hhhhp} imply the following lemma.
\begin{lemma}\label{notabb}
With respect to the adjoint action of $\varphi_6(D_6)$ or $\varrho_6(D_6)$,
all
highest weight vectors in $E_7$
having weight zero are of the form  $\alpha Y' +\beta X''' +\gamma  H$, for $\alpha, \beta, \gamma \in \mathbb{C}$, not all zero.  
\end{lemma}
\begin{lemma}\label{notabbc}
$[\alpha Y' +\beta X''' +\gamma  H]_{\varphi_6(D_6)}$  and  $[\alpha Y' +\beta X''' +\gamma  H]_{\varrho_6(D_6)}$ are abelian  subalgebras of $E_7$.   Further, any other subspace in the decomposition  of Eq. \eqref{hhhh} or \eqref{hhhhp}, or direct sums of such subspaces, is not an abelian subalgebra of $E_7$.
\end{lemma}
\begin{proof}
$[\alpha Y' +\beta X''' +\gamma  H]_{\varphi_6(D_6)}$  and  $[\alpha Y' +\beta X''' +\gamma  H]_{\varrho_6(D_6)}$ are $1$-dimensional, and hence necessarily abelian.  

The subspaces $[X'']_{\varphi_6(D_6)}, [X'']_{\varrho_6(D_6)}$, $[Y_1]_{\varphi_6(D_6)}$, and  $[Y_1]_{\varrho_6(D_6)}$  are not abelian by Table \ref{exb} since $\dim(V_{D_6}(\lambda_5))$ $=$   $\dim(V_{D_6}(\lambda_6))$ $=$$32>27$.  

$[X']_{\varphi_6(D_6)}$ and $[X']_{\varrho_6(D_6)}$ are isomorphic to $D_6$ and hence not abelian.  Noting $[Y', X'''] \neq 0, [Y',H] \neq 0$, and $[X''',H] \neq 0$  completes the proof.
\end{proof} 
For $\alpha, \beta, \gamma \in \mathbb{C}$, not all zero, note the following isomorphisms:
\begin{equation}\label{labbb}
\begin{array}{llllllll}
\displaystyle [\alpha Y' +\beta X''' +\gamma  H]_{\varphi_6(D_6)} &\cong& \mathbb{C},\\
\displaystyle [\alpha Y' +\beta X''' +\gamma  H]_{\varrho_6(D_6)}  &\cong&  \mathbb{C}.
\end{array}
\end{equation}
Lemma \ref{notabb} and \ref{notabbc} then give us the following proposition.

\begin{proposition}\label{mhhh}
The only abelian extension of $D_6$ that may be embedded into $E_7$ is the following:
\begin{equation}
\begin{array}{ll}
D_6 \inplus \mathbb{C}.
\end{array}\end{equation}
The (abelian) subspaces of $E_{7}$ which are isomorphic to ${\Bbb C}$  with respect to the adjoint action 
of $\varphi_6(D_6)$ or $\varrho_6(D_6)$ have (highest weight) vector  
$\alpha Y' +\beta X''' +\gamma  H$, for $\alpha, \beta, \gamma \in \mathbb{C}$, not all zero. 
\end{proposition}

\subsection{Abelian extensions of $D_5$ in $E_6$}

We begin by calculating the decomposition of $E_6$ with respect to the adjoint action of $\varphi_5(D_5)$ and $\varrho_5(D_5)$:
\begin{equation}\label{iiii}
\begin{array}{lllllllllll}
\displaystyle E_6 &\cong_{\varphi_5(D_5)}& [X']_{\varphi_5(D_5)}&\oplus& [Y_1]_{\varphi_5(D_5)}&\oplus& [X'']_{\varphi_5(D_5)} &\oplus&  \\
\displaystyle&& [H]_{\varphi_5(D_5)}\\
\displaystyle&\cong_{\varphi_5(D_5)}&  V_{D_5}(\lambda_2) &\oplus& V_{D_5}(\lambda_4) &\oplus& V_{D_5}(\lambda_5)&\oplus&  \\
\displaystyle &&V_{D_5}(0),
\end{array}
\end{equation}
\begin{equation}\label{iiiib}
\begin{array}{lllllllllll}
\displaystyle E_6 &\cong_{\varrho_5(D_5)}& [X']_{\varrho_5(D_5)}&\oplus& [Y_1]_{\varrho_5(D_5)}&\oplus& [X'']_{\varrho_5(D_5)} &\oplus & \\
\displaystyle &&[H]_{\varrho_5(D_5)}\\
\displaystyle &\cong_{\varrho_5(D_5)}&  V_{D_5}(\lambda_2) &\oplus& V_{D_5}(\lambda_5) &\oplus& V_{D_5}(\lambda_4)&\oplus & \\
\displaystyle &&V_{D_5}(0),
\end{array}
\end{equation}
where
\begin{equation}\label{ghyyy}
\begin{array}{llllll}
\displaystyle  X'&=&X_{6,5,4,3,2,4,5},\\ 
\displaystyle  X''&=&X_{6,5,4,2,3,1,4,3,5,4,2},\\ 
\displaystyle  H&=&2H_1+\frac{3}{2}H_2+\frac{5}{2}H_3+3H_4+2H_5+H_6.\\ 
\end{array}
\end{equation}
Note that 
\begin{equation} \label{iddd}
\begin{array}{lllll}
\varphi_5(D_5) &=& [X']_{\varphi_5(D_5)} &\cong& V_{D_5}(\lambda_2),\\
\varrho_5(D_5) &=& [X']_{\varrho_5(D_5)} &\cong& V_{D_5}(\lambda_2).
\end{array}
\end{equation}
 
Since $\varphi_5$ and $\varrho_5$ are the only embeddings of $D_5$ into $E_6$, up to equivalence (Theorem \ref{dymin}), 
determining which abelian extensions of $D_5$ may be embedded into $E_6$ amounts to determining which $E_6$-subspaces in Eq. \eqref{iiii} or \eqref{iiiib} (excluding $[X']_{\varphi_5(D_5)}$ and $[X']_{\varrho_5(D_5)}$), or direct sums of these subspaces, are abelian subalgebras of $E_6$.

The proofs of the following two lemmas proceed as in the above subsections and are omitted. 
\begin{lemma}\label{aballl}
The $E_6$-subspaces  $[ Y_1]_{\varphi_5(D_5)}$,   $[ X'']_{\varphi_5(D_5)}$, and $[ H ]_{\varphi_5(D_5)}$, as well as $[ Y_1]_{\varrho_5(D_5)}$,   $[ X'']_{\varrho_5(D_5)}$, and $[ H ]_{\varrho_5(D_5)}$ are abelian subalgebras of $E_6$.  
\end{lemma}

\begin{lemma}\label{notabbb}
The direct sum of any two or more of the subspaces in the decomposition  of Eq. \eqref{iiii} or \eqref{iiiib} is not an abelian subalgebra of $E_6$.
\end{lemma}

Note the following isomorphism of $D_5$-modules, with respect to the adjoint action of $\varphi_5(D_5)$ or $\varrho_5(D_5)$, respectively:
\begin{equation}\label{labbbb}
\begin{array}{ll}
\displaystyle [H]_{\varphi_5(D_5)} \cong   [H]_{\varrho_5(D_5)} \cong  \mathbb{C}.
\end{array}
\end{equation}
Lemmas \ref{aballl} and \ref{notabbb} then give the following proposition.
\begin{proposition}\label{mhhhh}
The only abelian extensions of $D_5$ that may be embedded into $E_6$ are the following:
\begin{equation}
\begin{array}{ll}
D_5 \inplus \mathbb{C},  ~ D_5 \inplus V_{D_5}(\lambda_4), ~ \text{and} ~ D_5 \inplus V_{D_5}(\lambda_5).
\end{array}
\end{equation}
The highest weight vectors of $D_{5}$-invariant abelian subspaces of $E_6$ are as follows:  
The only abelian subspace which is isomorphic to ${\Bbb C}$ 
with respect to the adjoint action of $\varphi_5(D_5)$ or $\varrho_5(D_5)$ has 
(highest weight) vector $\alpha  H$, for $\alpha \in \mathbb{C}^*$.

The abelian subspaces of  $E_6$ which carry $V_{D_5}(\lambda_5)$ irreps of 
$D_{5}$ with respect to the adjoint 
action of $\varphi_5(D_5)$ or $\varrho_5(D_5)$  have highest weight vectors
$\alpha  X''$
or $\alpha  Y_1$, respectively, for $\alpha \in \mathbb{C}^*$.  
The abelian subspaces of  $E_6$ which carry 
$V_{D_5}(\lambda_4)$ irreps of 
$D_{5}$ with respect to the adjoint 
action of $\varphi_5(D_5)$ or $\varrho_5(D_5)$ have highest weight vectors
$\alpha  Y_1$
or $\alpha  X''$, respectively, for $\alpha \in \mathbb{C}^*$.  
\end{proposition}

\section{Classifying the embeddings of the abelian extensions of  ${D}_n$ into $E_{n+1}$}\label{sec7}
 
In this section, we classify, up to inner automorphism,
 the embeddings identified in Section \ref{sec6} of each abelian extension of $D_n$ that may be embedded into $E_{n+1}$.  The classifications are  summarized in Theorems \ref{classaa}, \ref{classaab}, and \ref{classaaddd}. 

\subsection{Embeddings of the abelian extensions of  ${D}_7$ into $E_{8}$}

We begin by constructing lifts of the natural embedding $\varphi_7$  to $D_7 \inplus \mathbb{C}^{14}$, and to $D_7 \inplus \mathbb{C}^{}$.  
Let $u$ be a highest weight vector of the $D_7$-representation $V_{D_7}(\lambda_1)\cong \mathbb{C}^{14}$.  Then, a lift of $\varphi_7$ to $D_7 \inplus \mathbb{C}^{14}$ 
is uniquely determined by its evaluation on $u$.  By Proposition \ref{mhh}, the following then define all possible lifts of the  embedding $\varphi_7$ to $D_7 \inplus \mathbb{C}^{14}$  for each $\alpha \in \mathbb{C}^*$:
\begin{equation}\label{jh}
\begin{array}{cccccccccccccccc}
\begin{array}{cccccccccccc}
\widetilde{\varphi_7}^{\lambda_1,  \alpha}: &D_7 \inplus \mathbb{C}^{14} &\hookrightarrow & E_8\\
&u & \mapsto & \alpha X''', \\
\widetilde{\varphi_7}^{\lambda'_1,  \alpha}: &D_7 \inplus \mathbb{C}^{14} &\hookrightarrow & E_8\\
&u & \mapsto & \alpha Y'. 
\end{array} 
\end{array}
\end{equation}
Similarly, the following defines  all possible lifts of the  embedding $\varphi_7$ to $D_7 \inplus \mathbb{C}^{}$  for each $\alpha \in \mathbb{C}^*$.  
Here $u$ is any nonzero vector in $V_{D_7}(0) \cong \mathbb{C}$,  necessarily 
of highest weight.
\begin{equation}\label{jhg}
\begin{array}{cccccccccc}
\widetilde{\varphi_7}^{0, \alpha}: &D_7 \inplus \mathbb{C}^{} &\hookrightarrow & E_8\\
&u & \mapsto & \alpha H.
\end{array}
\end{equation}
The embeddings of Eqs. \eqref{jh} and \eqref{jhg} form a basis for the classification of abelian extension of $D_7$ into $E_8$, up to equivalence, described in Theorem \ref{classaa}.  We first present three useful lemmas.

\begin{lemma}\label{pa}
The embeddings $\widetilde{\varphi_7}^{\lambda_1, 1}$ and  $\widetilde{\varphi_7}^{\lambda'_1, 1}$ are not equivalent embeddings of $D_7 \inplus \mathbb{C}^{14}$ into $E_8$.
\end{lemma}
\begin{proof}
By way of contradiction, suppose $\widetilde{\varphi_7}^{\lambda_1,1} \sim \widetilde{\varphi_7}^{\lambda'_1,1}$.    Let $\rho: E_8 \rightarrow E_8$ be an inner automorphism of $E_8$ such that 
$\rho \circ \widetilde{\varphi_7}^{\lambda_1,1} = \widetilde{\varphi_7}^{\lambda'_1,1}$, which implies 
\begin{equation}\label{pqa}
\begin{array}{llllllllll}
\rho(X''')&=& Y'.
\end{array}
\end{equation}
The automorphism $\rho$ must send highest weight vectors to highest weight vectors of equal weight with respect to the adjoint action of $\varphi(D_7)$, hence 
\begin{equation}\label{kcx}
\begin{array}{llllllllll}
\rho(X'')&=& \alpha X'', & \rho(H) = \beta H, 
\end{array}
\end{equation}
for 
some 
$\alpha, \beta \in \mathbb{C}^*$.  Eq. \eqref{kcx} yields
\begin{equation}
\begin{array}{lllllll}
\displaystyle \rho([H,X''])&=&\rho(X'')&=&\alpha X'',\\
\displaystyle [\rho(H),\rho(X'')]&=& \alpha \beta X''.
\end{array}
\end{equation}
Hence, $\beta=1$ so that $\rho(H)=H$.  Eq. \eqref{pqa} then yields
\begin{equation}
\begin{array}{lllllll}
\displaystyle \rho([H,X'''])&=&\rho(2X''')&=&2Y',\\
\displaystyle [\rho(H),\rho(X''')]&=& [H,Y']&=&-2Y',
\end{array}
\end{equation}
a contradiction to $\rho$ being a Lie algebra homomorphism.
  \end{proof}

\begin{lemma}\label{paa}
The embeddings $\widetilde{\varphi_7}^{\lambda_1,1}$ and  $\widetilde{\varphi_7}^{\lambda_1,\alpha}$  are equivalent for all $\alpha \in \mathbb{C}^*$.  The embeddings $\widetilde{\varphi_7}^{\lambda'_1,1}$ and  $\widetilde{\varphi_7}^{\lambda'_1,\alpha}$  are equivalent for all $\alpha \in \mathbb{C}^*$.   
\end{lemma}
\begin{proof}
 For $\beta \in \mathbb{C}^*$, define an inner automorphism $\rho_\beta:E_8 \rightarrow E_8$ as follows:
\begin{equation}
\begin{array}{llllllllllll}
& H_{1} & \mapsto & H_{1}, & X_{1} & \mapsto & \beta X_{1},& Y_{1} & \mapsto & \frac{1}{\beta}Y_{1},\\
& H_{i} & \mapsto & H_{i}, & X_{i} & \mapsto &  X_{i},& Y_{i} & \mapsto & Y_{i},\\
\end{array}
\end{equation}
where $2\leq i\leq 8$.  We then have $\rho_{\beta} \circ \widetilde{\varphi_7}^{\lambda_1,1}=\widetilde{\varphi_7}^{\lambda_1,\alpha}$, $\rho_{\beta} \circ \widetilde{\varphi_7}^{\lambda'_1,\alpha}=\widetilde{\varphi_7}^{\lambda'_1,1}$, for $\beta^2 = \alpha$.  Hence $\widetilde{\varphi_7}^{\lambda_1,1}\sim \widetilde{\varphi_7}^{\lambda_1,\alpha}$ and $\widetilde{\varphi_7}^{\lambda'_1,1}\sim \widetilde{\varphi_7}^{\lambda'_1,\alpha}$.
\end{proof}

\begin{lemma}\label{paaa}
The embeddings $\widetilde{\varphi_7}^{0, \alpha}$ and  $\widetilde{\varphi_7}^{0, \beta}$ are equivalent if and only if $\alpha=\beta$. 
\end{lemma}
\begin{proof}
Let $\widetilde{\varphi_7}^{0, \alpha} \sim \widetilde{\varphi_7}^{0, \beta}$, and let $\rho$ be an automorphism of 
$E_8$ such that $\rho \circ \widetilde{\varphi_7}^{0, \alpha} = \widetilde{\varphi_7}^{0, \beta}$, which implies
\begin{equation}\label{zz1}
\begin{array}{lll}
\rho(H) = \frac{\beta}{\alpha} H=& \\
\frac{\beta}{\alpha}(4H_1+5 H_2+ 7H_3 +10H_4 + 
 8H_5  + 6H_6 
 +4H_7 +  2H_8).
 \end{array}
\end{equation}
Since $\rho$ is a Lie algebra homomorphism, and $\rho(H_i)=H_i$ for $2 \leq i \leq 8$, we also have
\begin{equation}\label{zz2}
\begin{array}{lllll}
\rho(H) = &\\4 \rho(H_1)+5 H_2+ 7H_3 +10H_4 + 
 8H_5  + 6H_6 
 +4H_7 +  2H_8.
\end{array}
\end{equation}
Eqs. \eqref{zz1} and \eqref{zz2} then yield
\begin{equation}\label{zz3}
\begin{array}{llllllllll}
\rho(H_1)= \frac{\beta}{\alpha} H_1 + \\ \frac{(\frac{\beta}{\alpha} -1)}{4}(5 H_2+ 7H_3 +10H_4 + 
 8H_5  + 6H_6 
 +4H_7 +  2H_8).  
 \end{array}
\end{equation}
The automorphism $\rho$ must send   highest weight vectors, with respect to the adjoint action of $\varphi_7(D_7)$,  to highest weight vectors of the same weight.    Hence, there is a $\gamma \in \mathbb{C}^*$ such that
\begin{equation}\label{zz4}
\rho(Y_1) = \gamma Y_1.
\end{equation}
Using Eq. \eqref{zz3},
\begin{equation}\label{zz5}
\begin{array}{llllllllll}
\displaystyle \rho( [H_1, Y_1])&=&\rho( -2Y_1)&=& -2\gamma Y_1,\\
 \displaystyle [\rho(H_1),\rho(Y_1)  ]&=& \gamma [\rho(H_1),Y_1]&= &-\frac{\gamma}{4}(\frac{\beta}{\alpha}+7)Y_1.
\end{array}
\end{equation}
Since $\rho$ is a Lie algebra homomorphism, Eq. \eqref{zz5} implies $\frac{\beta}{\alpha}=1$.  Hence, if $\widetilde{\varphi_7}^{0, \alpha} \sim \widetilde{\varphi_7}^{0, \beta}$, then $\alpha=\beta$.  The opposite implication is obvious.
\end{proof}

\begin{theorem}\label{classaa}
The only abelian extensions of $D_7$ that may be embedded into $E_8$ are 
$
D_7 \inplus V_{D_{7}}(\lambda_{1}) \cong 
D_7 \inplus \mathbb{C}^{14}$ 
and 
$
D_7 \inplus V_{D_{7}}(0) \cong 
D_7 \inplus \mathbb{C}^{}$.  
The embeddings $\widetilde{\varphi_7}^{\lambda_1, 1}$ and $\widetilde{\varphi_7}^{\lambda'_1,1}$ are a complete set of inequivalent embeddings of $D_7 \inplus \mathbb{C}^{14}$ into $E_8$.  
There is an infinite family of inequivalent embeddings of $D_7 \inplus \mathbb{C}$ into $E_8$, which contains all embeddings.  The infinite family is parameterized by a single continuous parameter:  $\widetilde{\varphi_7}^{0, \alpha}$, $\alpha \in \mathbb{C}^*$.  These results are displayed in Table \ref{end}.
\end{theorem}
\begin{proof}
The only abelian extensions of $D_7$ that may be embedded into $E_8$ are $D_7 \inplus \mathbb{C}^{14}$, and $D_7 \inplus \mathbb{C}^{}$ (a repetition of Proposition \ref{mhh}).  

Any embedding of an abelian extension of $D_7$ into $E_8$ restricted to $D_7$ is equivalent to $\varphi_7$  by Theorem \ref{dymin}.  Further, Eqs. \eqref{jh} and \eqref{jhg} list all possible lifts of $\varphi_7$  to the permissible abelian extensions of $D_7$.  The result then follows from Lemmas \ref{pa}, \ref{paa}, and \ref{paaa}.
\end{proof}

\begin{table} [!h]\renewcommand{\arraystretch}{1.6} \caption{Classification of embeddings of abelian extensions of $D_7$ into $E_8$ up to equivalence. } \label{end}\begin{center}
\begin{tabular}{|c|c|c|clclcl} 
\hline
Abelian Extension $D_7 \inplus V$ &  Embedding $D_7 \inplus V \hookrightarrow E_8$\\
\hline \hline
$
D_7 \inplus V_{D_{7}}(\lambda_{1}) \cong 
D_7 \inplus \mathbb{C}^{14}$& $\widetilde{\varphi_7}^{\lambda_1,1}$, ~~   $\widetilde{\varphi_7}^{\lambda'_1, 1}$\\ 
\hline 
$
D_7 \inplus V_{D_{7}}(0) \cong 
D_7 \inplus \mathbb{C}^{}$ &  $\widetilde{\varphi_7}^{0, \alpha}$, ~~ $\alpha \in \mathbb{C}^*$\\
\hline
\end{tabular}\end{center}
\end{table}

\begin{remark}
The Euclidean algebra $\mathfrak{e}(n)$ of isometries of $n$-dimensional Euclidean space is defined as 
$\mathfrak{e}(n) \cong \mathfrak{so}(n,\mathbb{C}) \inplus \mathbb{C}^n$.  
We thus have 
$\mathfrak{e}(14) \cong D_7 \inplus \mathbb{C}^{14}$.
Hence, one consequence of Theorem \ref{classaa}  is a classification of the embeddings of the Euclidean algebra $\mathfrak{e}(14)$ into $E_8$, up to equivalence.
\end{remark}

\subsection{Embeddings of the abelian extensions of  ${D}_6$ into $E_{7}$}

We begin by constructing lifts of the embeddings $\varphi_6: D_6 \hookrightarrow E_7$ and $\varrho_6: D_6 \hookrightarrow E_7$ 
to $
D_6 \inplus V_{D_{6}}(0) \cong 
D_6 \inplus \mathbb{C}^{}$.  Let $u$ be a 
(highest weight) 
vector of the $D_6$-representation $V_{D_6}(0)\cong \mathbb{C}^{}$.  By Proposition \ref{mhhh}, the following then define all possible lifts of the  embeddings $\varphi_6$ and $\varrho_6$  to $D_6 \inplus \mathbb{C}^{}$,  for  $\alpha, \beta, \gamma \in \mathbb{C}$, not all zero:

\begin{equation}\label{ggyy}
\begin{array}{lll}
\begin{array}{ccccccccccc}
\widetilde{\varphi_6}^{0,(\alpha,\beta, \gamma)}: &D_6 \inplus \mathbb{C}^{} &\hookrightarrow & E_7\\
&u & \mapsto &  \alpha Y' + \beta X''' + \gamma H,
\end{array}\\
\begin{array}{cccccccccccc}
\widetilde{\varrho_6}^{0,(\alpha,\beta, \gamma)}: &D_6 \inplus \mathbb{C}^{} &\hookrightarrow & E_7\\
&u & \mapsto &  \alpha Y' + \beta X''' + \gamma H.
\end{array}
\end{array}
\end{equation}

The embeddings of Eq. \eqref{ggyy} form a basis for the classification of abelian 
extensions
of $D_6$ into $E_7$, up to equivalence, described in Theorem \ref{classaab}.  We first present a useful lemma.

\begin{lemma}\label{pj2}
Let $\rho$ be an automorphism of $E_7$  which fixes $X_i, Y_i$, and $H_i$ for $2 \leq i \leq 7$.  Then
\begin{equation}
\begin{array}{lllllllllll}
\rho(Y') &=& c^2Y'-cdH-d^2X''', \\
\rho(X''') &=& a^2X'''+abH-b^2Y',\\
 \rho(H)&=&(ac+bd)H+2adX'''-2bcY', 
\end{array}
\end{equation}
where $a, b, c, d \in \mathbb{C}$ such that $ac-bd=1$.
\end{lemma}
\begin{proof}
Let $\rho: E_7 \rightarrow E_7$ be an automorphism of $E_7$ which  fixes $X_i, Y_i$, and $H_i$ for $2 \leq i \leq 7$.  
A highest weight vector of $E_7$ must be sent to a highest weight vector of $E_7$ of equal weight with respect to the adjoint action of $\varphi_6(D_6)$.  Hence, referring to Eq. \eqref{hhhh},
\begin{equation}\label{u3}
\begin{array}{llllllll}
\rho(Y_1) &=& c Y_1 + d X'',  &
\rho(X'') &=& a X'' +b Y_1.
\end{array}
\end{equation}
Computation shows that there exists a sequence $a_1, a_2,...,a_k$, where $2\leq a_i \leq 7$ for each $i$, such that
\begin{equation}\label{gfd}
\begin{array}{lllll}
\displaystyle [[[[[X'', Y_{a_1}],Y_{a_2}],Y_{a_3}]\cdots], Y_{a_k}]  &=&-X_1, ~\text{and}\\ 
\displaystyle  [[[[[Y_1, Y_{a_1}],Y_{a_2}],Y_{a_3}]\cdots], Y_{a_k}]  &=&-Y_{7,6,5,4,3,2,4,5,6,1,3,4,5,2,4,3}. 
 \end{array}
\end{equation}
Let us define $Y''= Y_{7,6,5,4,3,2,4,5,6,1,3,4,5,2,4,3}$.   Note that the elements $X_1$ and $Y''$ are lowest weight vectors with respect to the adjoint action of $\varphi_6(D_6)$.  Equations  \eqref{u3} and \eqref{gfd}, together with the fact that $\rho(Y_{a_i})=Y_{a_i}$ for $2\leq a_i \leq 7$,  then imply
\begin{equation}\label{u4}
\begin{array}{llllllllll}
\rho(Y'') &=& c Y'' +d X_1, & 
\rho(X_1) &=& aX_1 +b Y''.
\end{array}
\end{equation}
Note that $[X_1,X'']=X'''$ and  $[Y_1,Y'']=Y'$, so that Eqs. \eqref{u3} and \eqref{u4} imply
\begin{equation}\label{ut4}
\begin{array}{llllllllll}
\rho(X''') &=& a^2X'''+abH-b^2Y', \\
\rho(Y') &=& c^2Y'-cdH-d^2X'''.
\end{array}
\end{equation}
Note that
\begin{equation}\label{asxt555}
\begin{array}{lllll}
\displaystyle [\rho(Y_1), \rho(X''')]&=&[cY_1+dX'',a^2X'''+abH-b^2Y']\\
\displaystyle &=&b(ac-bd)Y_1+a(ac-bd)X''.
\end{array}
\end{equation}
Since $[Y_1,X''']=X''$, Eqs. \eqref{u3} and \eqref{asxt555} imply $b(ac-bd)=b$ and $a(ac-bd)=a$.  Since not both $a=0$ and $b=0$, we have
\begin{equation}
ac-bd=1.
\end{equation}
Since $[X''',Y']=H$ and $ac-bd=1$, Eq. \eqref{ut4} implies
\begin{equation}\label{asxt55}
\begin{array}{lllll}
\displaystyle \rho(H)=(ac+bd)H+2adX'''-2bcY'.
\end{array}
\end{equation}
\end{proof}

\begin{theorem}\label{classaab}
The only abelian extension of $D_6$ that may be embedded into $E_7$ is 
$
D_6 \inplus V_{D_{6}}(0)  \cong 
D_6 \inplus \mathbb{C}^{}$. 
Any embedding of $D_6 \inplus \mathbb{C}$ into $E_7$ is equivalent to $\widetilde{\varphi_6}^{0,(\alpha,\beta, \gamma)}$ or $\widetilde{\varrho_6}^{0,(\alpha,\beta, \gamma)}$ for some $\alpha, \beta, \gamma \in \mathbb{C}$, not all zero.  For $\alpha, \beta, \gamma \in \mathbb{C}$, not all zero, the embeddings are classified according to the following rules:

\begin{enumerate}[(a)]
\item  $\widetilde{\varphi_6}^{0,(\alpha,\beta, \gamma)} \nsim \widetilde{\varrho_6}^{0,(\alpha,\beta, \gamma)} ~\text{for all}~~ \alpha, \beta, \gamma$.
\item  $\widetilde{\varphi_6}^{0,(\alpha,\beta, \gamma)} \sim \widetilde{\varphi_6}^{0,(\alpha',\beta', \gamma')}\Leftrightarrow$\\
$  (\alpha', \beta', \gamma' ) =  
 (\alpha  c^2-\beta b^2-2\gamma bc, -\alpha d^2+\beta a^2+2\gamma ad, -\alpha cd +\beta ab+\gamma (ac+bd)), ~\text{for some}~ a,b,c,d \in \mathbb{C},~\text{such that}~ ac-bd=1$.
\item $\widetilde{\varrho_6}^{0,(\alpha,\beta, \gamma)} \sim \widetilde{\varrho_6}^{0,(\alpha',\beta', \gamma')}\Leftrightarrow$ \\
$  (\alpha', \beta', \gamma' ) =  
 (\alpha  c^2-\beta b^2-2\gamma bc, -\alpha d^2+\beta a^2+2\gamma ad, -\alpha cd +\beta ab+\gamma (ac+bd)), ~\text{for some}~ a,b,c,d \in \mathbb{C},~\text{such that}~ ac-bd=1$.
\end{enumerate}
\end{theorem}
\begin{proof}
Let $\widetilde{\varphi_6}^{0,(\alpha,\beta, \gamma)} \sim \widetilde{\varphi_6}^{0,(\alpha',\beta', \gamma')}$, and let $\rho$ be an automorphism of $E_7$ such that 
\begin{equation}\label{u10}
\rho \circ \widetilde{\varphi_6}^{0,(\alpha,\beta, \gamma)} = \widetilde{\varphi_6}^{0,(\alpha',\beta', \gamma')}.
\end{equation}
Eq. \eqref{u10} implies
\begin{equation}\label{u12}
\rho(\alpha Y' + \beta X''' + \gamma H ) = \alpha' Y' + \beta' X''' + \gamma' H.
\end{equation}
Eq. \eqref{u10} implies that $\rho$ fixes $X_i, Y_i$, and $H_i$ for $2 \leq i \leq 7$.  Hence Lemma \ref{pj2} implies 
\begin{equation}\label{u14b}
\begin{array}{llllllll}
 \rho(\alpha Y' + \beta X''' + \gamma H ) = \alpha\rho(Y') + \beta\rho(X''') + \gamma \rho(H)=\\ (\alpha  c^2-\beta b^2-2\gamma bc) Y'
 +(-\alpha d^2+\beta a^2+2\gamma ad)X'''+\\( -\alpha cd +\beta ab+\gamma (ac+bd))H, 
\end{array}
\end{equation}
for some $a, b, c, d \in \mathbb{C}$ such that $ac-bd=1$.  Hence, Eqs. \eqref{u12} and \eqref{u14b} imply $(\alpha', \beta', \gamma' ) =$  $(\alpha  c^2-\beta b^2-2\gamma bc, -\alpha d^2+\beta a^2+2\gamma ad, -\alpha cd +\beta ab+\gamma (ac+bd))$.  

We now prove the opposite implication of item (b).  Let $\rho: E_7 \rightarrow E_7$ be the (inner) automorphism of $E_7$ that fixes $X_i, Y_i$, and $H_i$ for $2 \leq i \leq 7$, and 
\begin{equation}\label{u14c}
\begin{array}{llllllll}
\rho(X_1)&=& aX_1+bY'',\\
\rho(Y_1)&=& cY_1+dX'',\\
\rho(H_1)&=& H_1+bdH+adX'''-bcY'.
\end{array}
\end{equation}
Then, 
\begin{equation}
\begin{array}{lllllllllll}
\rho(Y') &=& c^2Y'-cdH-d^2X''', \\
\rho(X''') &=& a^2X'''+abH-b^2Y',\\
 \rho(H)&=&(ac+bd)H+2adX'''-2bcY', 
\end{array}
\end{equation}
so that 
\begin{equation}\label{u14bbbb}
\begin{array}{llllllll}
 \rho(\alpha Y' + \beta X''' + \gamma H ) = \alpha' Y'+\beta'X'''+\gamma'H,
\end{array}
\end{equation}
as in Eq. \eqref{u14b}, which implies $\rho \circ \widetilde{\varphi_6}^{0,(\alpha,\beta, \gamma)} = \widetilde{\varphi_6}^{0,(\alpha',\beta', \gamma')}$.  Hence, we have established item (b).  Item (c) is proved in a similar manner.  

Now we consider item (a).  Since $\varphi_6 \nsim \varrho_6$ by Theorem \ref{dymin}, then  $\widetilde{\varphi_6}^{0,(\alpha,\beta, \gamma)}$ $\nsim$ $\widetilde{\varrho_6}^{0,(\alpha',\beta', \gamma')}$.
\end{proof}

\subsection{Embeddings of the abelian extensions of  ${D}_5$ into $E_{6}$}

We begin by constructing lifts of the embeddings $\varphi_5$ and $\varrho_5$  to $D_5 \inplus V_{D_5}(\lambda_5)$, $D_5 \inplus V_{D_5}(\lambda_4)$, and to 
$
D_5 \inplus V_{D_{5}}(0) \cong 
D_5 \inplus \mathbb{C}^{}$.  
Let $u$ be a highest weight vector of the $D_5$-representation $V_{D_5}(\lambda_5)$, $V_{D_5}(\lambda_4)$,  or $\mathbb{C}$, respectively.  By Proposition \ref{mhhhh}, the following then define all possible lifts of  $\varphi_5$ or $\varrho_5$ to $D_5 \inplus V_{D_5}(\lambda_5)$, $D_5 \inplus V_{D_5}(\lambda_4)$, and to $D_5 \inplus \mathbb{C}^{}$, respectively, for each $\alpha \in \mathbb{C}^*$:
\begin{equation}\label{qqwwe}
\begin{array}{llllllllllllllllll}
\begin{array}{ccccccccccccc}
\widetilde{\varphi_5}^{\lambda_5, \alpha}: &D_5 \inplus V_{D_5}(\lambda_5) &\hookrightarrow & E_6\\
&u & \mapsto & \alpha X'',&
\\
\widetilde{\varrho_5}^{\lambda_5, \alpha}: &D_5 \inplus V_{D_5}(\lambda_5) &\hookrightarrow & E_6\\
&u & \mapsto & \alpha  Y_1,&
\\
\widetilde{\varphi_5}^{\lambda_4, \alpha}: &D_5 \inplus V_{D_5}(\lambda_4) &\hookrightarrow & E_6\\
&u & \mapsto & \alpha Y_1,&
\\
\widetilde{\varrho_5}^{\lambda_4, \alpha}: &D_5 \inplus V_{D_5}(\lambda_4) &\hookrightarrow & E_6\\
&u & \mapsto & \alpha X'',&
\\
\widetilde{\varphi_5}^{0, \alpha}: &D_5 \inplus \mathbb{C} &\hookrightarrow & E_6\\
&u & \mapsto & \alpha H,&
\\
\widetilde{\varrho_5}^{0, \alpha}: &D_5 \inplus \mathbb{C} &\hookrightarrow & E_6\\
&u & \mapsto & \alpha H.&
\end{array}
\end{array}
\end{equation}
The embeddings of  \eqref{qqwwe} form the basis of the classification of abelian extension of $D_5$ into $E_6$, up to equivalence, in Theorem \ref{classaaddd}.  We first prove two useful lemmas.

\begin{lemma}\label{paacc}
For all $\alpha \in \mathbb{C}^*$, 
\begin{equation}
\begin{array}{lllllllllll}
\widetilde{\varphi_5}^{\lambda_5,1}&\sim&\widetilde{\varphi_5}^{\lambda_5,\alpha}, &  \widetilde{\varrho_5}^{\lambda_5,1}&\sim&\widetilde{\varrho_5}^{\lambda_5,\alpha},  \\
 \widetilde{\varphi_5}^{\lambda_4,1}&\sim&\widetilde{\varphi_5}^{\lambda_4,\alpha}, &
  \widetilde{\varrho_5}^{\lambda_4,1}&\sim&\widetilde{\varrho_5}^{\lambda_4,\alpha},\\
    \widetilde{\varphi_5}^{\lambda_4,1}&\nsim&\widetilde{\varrho_5}^{\lambda_4,1}, &  \widetilde{\varphi_5}^{\lambda_5,1}&\nsim&\widetilde{\varrho_5}^{\lambda_5,1}.\\
\end{array}
\end{equation}
\end{lemma}
\begin{proof}
 Define an inner automorphism  $\rho:E_6 \rightarrow E_6$ as follows:
\begin{equation}
\begin{array}{llllllllllll}
& H_{1} & \mapsto & H_{1},& X_{1} & \mapsto & \alpha X_{1}, & Y_{1} & \mapsto & \frac{1}{\alpha} Y_{1}, \\
& H_{i} & \mapsto & H_{i},& X_{i} & \mapsto &  X_{i}, & Y_{i} & \mapsto &  Y_{i},
\end{array}
\end{equation}
where $2\leq i\leq 6$.  We then have 
\begin{equation}
\begin{array}{lllllllllllll}
\rho_{} \circ \widetilde{\varphi_5}^{\lambda_5,1}&=&\widetilde{\varphi_5}^{\lambda_5,\alpha}, & \rho_{} \circ \widetilde{\varrho_5}^{\lambda_5,\alpha}&=&\widetilde{\varrho_5}^{\lambda_5,1}\\
\rho_{} \circ \widetilde{\varphi_5}^{\lambda_4,\alpha}&=&\widetilde{\varphi_5}^{\lambda_4,1}, & \rho_{} \circ \widetilde{\varrho_5}^{\lambda_4,1}&=&\widetilde{\varrho_5}^{\lambda_4,\alpha}.\\
\end{array}
\end{equation}
Since $\varphi_5 \nsim \varrho_5$ by Theorem \ref{dymin}, then $\widetilde{\varphi_5}^{\lambda_4,1}\nsim\widetilde{\varrho_5}^{\lambda_4,1}$, and
  $\widetilde{\varphi_5}^{\lambda_5,1}\nsim\widetilde{\varrho_5}^{\lambda_5,1}$. \end{proof}

\begin{lemma}\label{paaccc}
  The embeddings $\widetilde{\varphi_5}^{0, \alpha}$ and  $\widetilde{\varphi_5}^{0, \beta}$  are equivalent if and only if $\alpha = \beta$.    Similarly, the embeddings $\widetilde{\varrho_5}^{0, \alpha}$ and  $\widetilde{\varrho_5}^{0, \beta}$  are equivalent if and only if $\alpha = \beta$.   Finally,  $\widetilde{\varphi_5}^{0, \alpha}\nsim \widetilde{\varrho_5}^{0, \beta}$, for all $\alpha, \beta$.
\end{lemma}
\begin{proof}
Let $\widetilde{\varphi_5}^{0, \alpha} \sim \widetilde{\varphi_5}^{0, \beta}$, and let $\rho$ be an automorphism of 
$E_6$ such that $\rho \circ \widetilde{\varphi_5}^{0, \alpha} = \widetilde{\varphi_5}^{0, \beta}$, which implies
\begin{equation}\label{z1}
\rho(H) = \frac{\beta}{\alpha} H= \frac{\beta}{\alpha}
\biggl(
2H_1+\frac{3}{2}H_2+\frac{5}{2}H_3+3H_4+2H_5+H_6
\biggr).
\end{equation}
Since $\rho$ is a Lie algebra homomorphism, and $\rho(H_i)=H_i$ for $2 \leq i \leq 6$, we also have
\begin{equation}\label{z2}
\rho(H) = 2 \rho(H_1)+\frac{3}{2}H_2+\frac{5}{2}H_3+3H_4+2H_5+H_6 .
\end{equation}
Eqs. \eqref{z1} and \eqref{z2} then yield
\begin{equation}\label{z3}
\rho(H_1) = \frac{\beta}{\alpha} H_1+
\biggl(
\frac{\beta}{\alpha}-1
\biggr)
\biggl(
\frac{3}{4}H_2+\frac{5}{4}H_3+\frac{3}{2}H_4+H_5+\frac{1}{2}H_6 
\biggr).
\end{equation}
The automorphism must send  highest weight vectors, with respect to the adjoint action of $\varphi_5(D_5)$,  to highest weight vectors of the same weight.  Hence, for some $\gamma \in \mathbb{C}^*$,
\begin{equation}\label{z4}
\rho(Y_1) = \gamma Y_1.
\end{equation}
Considering Eqs. \eqref{z3} and \eqref{z4},
\begin{equation}\label{z5}
\begin{array}{llllllllll}
\displaystyle \rho( [H_1, Y_1])&=&\rho( -2Y_1)=-2 \gamma Y_1,\\
 \displaystyle [\rho(H_1),\rho(Y_1)  ]&=& -\frac{\gamma}{4} (3\frac{\beta}{\alpha}+5)Y_1.
\end{array}
\end{equation}
Since $\rho$ is a Lie algebra homomorphism, Eq. \eqref{z5} implies $\frac{\beta}{\alpha}=1$.  Hence, if $\widetilde{\varphi_5}^{0, \alpha} \sim \widetilde{\varphi_5}^{0, \beta}$, then $\alpha=\beta$.  The opposite implication is obvious.

Similarly we prove $\widetilde{\varrho_5}^{0, \alpha} \sim \widetilde{\varrho_5}^{0, \beta}$ if and only if $\alpha = \beta$.
 Finally, since $\varphi_5 \nsim \varrho_5$ by Theorem \ref{dymin}, then $\widetilde{\varphi_5}^{0,\alpha}\nsim\widetilde{\varrho_5}^{0,\beta}$ for all $\alpha, \beta$.
\end{proof}

\begin{theorem}\label{classaaddd}
The only abelian extensions of $D_5$ that may be embedded into $E_6$ are $D_5 \inplus V_{D_7}(\lambda_5)$, $D_5 \inplus V_{D_7}(\lambda_4)$, and $D_5 \inplus \mathbb{C}^{}$.  The embeddings $\widetilde{\varphi_5}^{\lambda_5, 1}$ and $\widetilde{\varrho_5}^{\lambda_5,1}$ are a complete set of inequivalent embeddings of $D_5 \inplus V_{D_7}(\lambda_5)$ into $E_6$.  The embeddings $\widetilde{\varphi_5}^{\lambda_4, 1}$ and $\widetilde{\varrho_5}^{\lambda_4,1}$ are a complete set of inequivalent embeddings of $D_5 \inplus V_{D_7}(\lambda_4)$ into $E_6$.  
There are two infinite families of inequivalent embeddings of $D_5 \inplus \mathbb{C}$ into $E_6$, which contain all
inequivalent embeddings.  Each family is parameterized by a single continuous parameter:   $\widetilde{\varphi_5}^{0, \alpha}$ and $\widetilde{\varrho_5}^{0,\alpha}$, for $\alpha \in \mathbb{C}^*$.   These results are displayed in Table \ref{enddd}.
\end{theorem}
\begin{proof}
The only abelian extensions of $D_5$ that may be embedded into $E_6$ are $D_5 \inplus V_{D_7}(\lambda_5)$, $D_5 \inplus V_{D_7}(\lambda_4)$, and $D_5 \inplus \mathbb{C}^{}$ (a repetition of Proposition \ref{mhhhh}).  

The restriction to $D_5$ of any embedding of an abelian extension of $D_5$ into $E_6$  is equivalent to $\varphi_5$ or $\varrho_5$ by Theorem \ref{dymin}.  Further, \eqref{qqwwe} lists all possible lifts of $\varphi_5$ or $\varrho_5$ to the permissible abelian extensions of $D_5$.  The result then follows from Lemmas \ref{paacc} and \ref{paaccc}.
\end{proof}

\begin{table} [!h]\renewcommand{\arraystretch}{1.6} \caption{Classification of embeddings of abelian extensions of $D_5$ into $E_6$ up to equivalence.} \label{enddd}\begin{center}
\begin{tabular}{|c|c|c|clclcl} 
\hline
Abelian Extension $D_5 \inplus V$ &  Embedding $D_5 \inplus V \hookrightarrow E_6$\\
\hline \hline
$D_5 \inplus V_{D_5}(\lambda_5)$& $\widetilde{\varphi_5}^{\lambda_5, 1}$, ~~$\widetilde{\varrho_5}^{\lambda_5, 1}$\\ 
\hline 
$D_5 \inplus V_{D_5}(\lambda_4)$& $\widetilde{\varphi_5}^{\lambda_4, 1}$, ~~$\widetilde{\varrho_5}^{\lambda_4, 1}$ \\ 
\hline 
$
D_5 \inplus V_{D_{5}}(0) \cong
D_5 \inplus \mathbb{C}^{}$ &  $\widetilde{\varphi_5}^{0, \alpha}$, ~~$\widetilde{\varrho_5}^{0, \alpha}$, ~~ $\alpha \in \mathbb{C}^*$\\
\hline
\end{tabular}\end{center}
\end{table}

\section{Representations of  $D_{n}\inplus V$ from representations of $E_{n+1}$}\label{sec8}

As an application of the above classification, we also consider the irreducible representations of $E_{n+1}$ restricted to $D_n \inplus V$, for each abelian extension and each embedding.  The results are summarized in Theorem \ref{ind} and Examples \ref{indaa} to \ref{indc} below.

\begin{theorem}\label{ind}
Any irreducible representation of $E_{6}$ restricts to an indecomposable representation of $D_{5}\inplus V_{D_5}(\lambda_4)$ or $D_{5}\inplus V_{D_5}(\lambda_5)$ with respect to any embedding of $D_{5}\inplus V_{D_5}(\lambda_4)$ or $D_{5}\inplus V_{D_5}(\lambda_5)$  into $E_6$, respectively.
\end{theorem}

\begin{proof}
The  theorem follows immediately from the observation that the image of $D_5 \inplus V_{D_5}(\lambda_4)$ or $D_5 \inplus V_{D_5}(\lambda_5)$ under any embedding into $E_6$, all of which were identified in Theorem \ref{classaaddd}, contains 
all positive or  all negative root vectors of $E_6$ (see for instance  \cite{dr3});  
thus its proof is 
omitted.
\end{proof}

 The following three examples show that for every case other than the two considered in Theorem \ref{ind}, an irreducible representation of 
 $E_{n+1}$ may decompose when restricted to an abelian extension $D_n \inplus V$, under any embedding.

\begin{example}\label{indaa}
Irreducible representations of $E_7$ may decompose when restricted to $D_6 \inplus \mathbb{C}$.  For instance, direct calculation gives us the following decomposition of the $E_7$ fundamental representation  $V_{E_7}(\lambda_1)$:
\begin{equation}
\begin{array}{llllllllllllll}
V_{E_7}(\lambda_1)|_{\widetilde{\varphi_7}^{0,(\alpha,\beta, \gamma)}} &\cong&
 V_{D_6} (\lambda_2) \oplus (V_{D_6}(\lambda_5) + V_{D_6}(\lambda_5)) \oplus \\
 &&  (V_{D_6}(0) + V_{D_6}(0)+ V_{D_6}(0)),\\
V_{E_7}(\lambda_1)|_{\widetilde{\varrho_7}^{0,(\alpha,\beta, \gamma)}} &\cong&
 V_{D_6} (\lambda_2) \oplus (V_{D_6}(\lambda_6) + V_{D_6}(\lambda_6))\oplus\\
  && (V_{D_6}(0) + V_{D_6}(0)+ V_{D_6}(0)).\\
\end{array}
\end{equation}
The decomposition  of $V_{D_6}(\lambda_5) + V_{D_6}(\lambda_5)$, $V_{D_6}(\lambda_6) + V_{D_6}(\lambda_6)$, and 
$V_{D_6}(0) + V_{D_6}(0)+ V_{D_6}(0)$  will depend on the values of $\alpha, \beta$, and $\gamma$.  
\end{example}

\begin{example}\label{inda}
The adjoint representation of $E_6$  decomposes when restricted to $D_5 \inplus \mathbb{C}$, with respect to any embedding.  Similarly, the adjoint 
representation
of $E_8$  decomposes when restricted to $D_7 \inplus \mathbb{C}$, with respect to any embedding.  In each case,  the decomposition with respect to $D_5 \inplus \mathbb{C}$ or $D_7 \inplus \mathbb{C}$ is the same as the decomposition with respect to the adjoint action of the image of $D_5$ or $D_7$, respectively. 

The reason for this decomposition is that the highest weight vector of $\mathbb{C}$, with respect to the adjoint action of the image of $D_5$ or $D_7$, is an element of a Cartan subalgebra of $E_6$ or $E_8$, respectively (see \eqref{iiii} and \eqref{iiiib} or \eqref{hhhhj}, respectively).  

More generally,  the decomposition of an $E_6$ or $E_8$ irrep with respect to $D_5 \inplus \mathbb{C}$ or $D_7 \inplus \mathbb{C}$, respectively, is the same as the decomposition with respect to the restricted action of the image of $D_5$ or $D_7$, respectively.

\end{example}

\begin{example}\label{indc}
 Irreducible representations of $E_8$ may decompose when restricted to $D_7 \inplus \mathbb{C}^{14}$.  
 For instance, direct calculation gives us the following decomposition of the $E_8$ fundamental representation $V_{E_8}(\lambda_8)$:
\begin{equation}
\begin{array}{llllllllllllllllll}
V_{E_8}(\lambda_8)|_{\widetilde{\varphi_7}^{\lambda_1,1}} \cong V_{E_8}(\lambda_8)|_{\widetilde{\varphi_7}^{\lambda'_1,1}} 
\cong\\( V_{D_7} (\lambda_2) + 2V_{D_7} (\lambda_1)+ V_{D_7}(0))
\oplus (V_{D_7} (\lambda_6) + V_{D_7} (\lambda_7)).
\end{array}
\end{equation}
Each set of parentheses on the right side of the last equation contains a representation of ${D}_7\inplus \mathbb{C}^{14}$, described in terms of the decomposition of its restriction to $D_7$.
\end{example}

\section{Conclusions}

We determined all abelian extensions of $D_n$ that may be embedded into the exceptional Lie algebra $E_{n+1}$, $n=5, 6$, and $7$.  We then classified each of these embeddings, up to equivalence.  The classifications are summarized in Theorems \ref{classaa}, \ref{classaab}, and \ref{classaaddd}.

As an application of the  classification, we also considered the irreducible representations of $E_{n+1}$ restricted to $D_n \inplus V$, for each abelian extension and each embedding.  Any irreducible representation of $E_{6}$ restricts to an indecomposable representation of $D_{5}\inplus V_{D_5}(\lambda_4)$ or $D_{5}\inplus V_{D_5}(\lambda_5)$ with respect to any embedding.  For every other case, an irreducible representation of $E_{n+1}$ restricted to $D_n \inplus V$ may decompose.

\section*{Acknowledgements}

The work of A.D. and D.K. is partially supported by research grants from the Professional Staff Congress/City University of New York.  The work 
of J.R. is partially supported by the Natural Sciences and Engineering Research Council.  
We thank Andrey Minchenko for a helpful conversation on embeddings of simple Lie algebras into exceptional Lie algebras.


\begin{thebibliography}{9}

\bibitem{baez}
Baez, J. C., ``The octonions."  Bull. Amer. Math. Soc., 39(2), 145-205 (2001).

\bibitem{hu}
  Carter, R., \emph{Lie Algebras of Finite and Affine Type},
  (Cambridge University Press, Cambridge, 2005).
  

\bibitem{dr1}
 Douglas, A., and Repka, J., ``Embeddings of  the Euclidean algebra $\mathfrak{e}(3)$ into  $\mathfrak{sl}(4,\mathbb{C})$ and restrictions of
 irreducible representations of $\mathfrak{sl}(4,\mathbb{C})$,''
  J. Math. Phys. 52, 013504 (2011).

\bibitem{dr2}
Douglas, A., and Repka, J., ``Indecomposable representations of the Euclidean algebra $\e$ from irreducible representations of $\mathfrak{sl}(4,\mathbb{C})$,'' Bull. Aust. Math. Soc., 83, 439-449 (2011). 

\bibitem{dr3}
Douglas, A., and Repka, J., ``Indecomposable representations of the Euclidean algebra $\e$ from irreducible representations of the symplectic algebra $\mathfrak{sp}(4,\mathbb{C})$,'' J. Phys.: Conf. Ser. 284, 012022 (2011). 


\bibitem{dynkin}
 Dynkin, E. B., ``Semisimple subalgebras of semisimple Lie algebras,"  Mat. Sbornik N.S., 30(72):  349-462 (1952).  English translation
 in: Amer. Math. Soc. Transl. (6),  111-244 (1957).
  
  \bibitem{fh}
Fulton, W, and Harris, J., \emph{Representation Theory},
  (Springer-Verlag, New York, 1991).
  
\bibitem{gap}
GAP-Groups, Algorithms, and Programming, Version 4.2, 2000
(http://www-gap.dcs.st-and.ac.uk/~gap).


\bibitem{humphreys}
  Humphreys, J. E., \emph{Introduction to Lie Algebras and Representation Theory},
  (Springer-Verlag, New York, 1972).

\bibitem{losev}
Losev, I., ``On invariants of a set of elements of a semisimple Lie algebra," J. Lie Theory 20, 17-30 (2010).



\bibitem{mal}
Maltsev, A. I., ``Commutative subalgebras of semisimple Lie algebra,"  Amer. Math. Soc. Transl. no. 40, (1951). 



\bibitem{min}
Minchenko, A. N., ``The semisimple subalgebras of exceptional Lie algebras,"  Trans. Moscow Math. Soc., S 0077-1554(06), 225-259 (2006).






\end{thebibliography}
\end{document}